\documentclass{article}

\usepackage{a4}

\usepackage{epsfig}

\usepackage{amsfonts}

\usepackage[french]{babel}
\usepackage[latin1]{inputenc}
\usepackage[T1]{fontenc}

\usepackage{upref}
\usepackage{fancyhdr,amsmath, graphicx, psfrag, color ,amsfonts,layout, amssymb}
\usepackage{euscript,epsfig}

\newtheorem{theo}{Th\'eor\`eme}[section]

\newtheorem{lemm}[theo]{Lemme}
\newtheorem{prop}[theo]{Proposition}
\newtheorem{fait}[theo]{Fait}

\newtheorem{rema}[theo]{Remarque}
\newtheorem{defi}[theo]{D\'efinition}

\numberwithin{equation}{section}
\newenvironment{proof}{{\flushleft \em D\'emonstration\,: }}{\hfill $\square$ \vspace{5mm}}

\def \R{\mathbb R}
\def \N{{\mathbb N}}

\def \Z{\mathbb Z}

\def \H{\mathbb H}

\def\D{\mathbb D}

\begin{document}

\title{Densit\'e de demi-horocycles sur une surface hyperbolique g\'eom\'etriquement infinie}

\author{Barbara Schapira}

\maketitle

\emph{LAMFA, Universit\'e Picardie Jules Verne, 33 rue St Leu, 
80000 Amiens, France}

\footnote{37A40,37A17, 37B99, 37D40}

\begin{abstract} On the unit tangent bundle of a hyperbolic surface, 
we study the density of positive orbits $(h^s v)_{s\ge 0}$ under the horocyclic flow. 
More precisely, given a full  orbit $(h^sv)_{s\in \R}$, 
we prove that under a weak assumption on the vector $v$, 
both half-orbits $(h^sv)_{s\ge 0}$ and $(h^s v)_{s\le 0}$ 
are simultaneously dense or not in the nonwandering set 
$\mathcal{E}$ of the horocyclic flow. 
We give also a counter-example to this result when this assumption is not satisfied. 
\end{abstract}



\section{Introduction}

Hedlund \cite{Hedlund} a d\'emontr\'e que sur le fibr\'e unitaire tangent d'une 
surface hyperbolique de volume fini, toutes les orbites positives non p\'eriodiques 
 $(h^sv)_{s\ge 0}$ du flot horocyclique  sont denses. 

Par ailleurs, on a maintenant un crit\`ere g\'eom\'etrique complet 
(voir \cite{Dalbo} pour un r\'esultat g\'en\'eral et des r\'ef\'erences) pour savoir si, 
sur une surface hyperbolique quelconque, 
un horocycle $(h^sv)_{s\in \R}$ est dense dans l'ensemble non 
errant $\mathcal{E}$ du flot horocyclique ou non. Il l'est 
si et seulement si $v$ est {\em horosph\'erique} (voir paragraphe 3). 

Une question naturelle du point de vue dynamique est de savoir si 
lorsqu'un tel horocycle complet est dense, 
les demi-horocycles $(h^sv)_{s\ge 0}$ et $(h^sv)_{s\le 0}$ le sont tous les deux. 

Hedlund avait un r\'esultat positif partiel, 
sur les surfaces dites <<~de premi\`ere esp\`ece~>>, 
pour les vecteurs {\em radiaux}, i.e. les vecteurs $v$ dont 
la g\'eod\'esique 
revient infiniment souvent dans un compact. 

Dans \cite{Schapira}, j'ai trait\'e le cas des surfaces hyperboliques 
g\'eom\'etriquement finies
(i.e. dont le groupe fondamental est de type fini), 
et j'ai montr\'e que la r\'eponse est 
toujours positive, i.e. les deux demi-horocycles $(h^sv)_{s\ge 0}$ et $(h^sv)_{s\le 0}$ 
sont toujours denses (ou non) simultan\'ement, 
hormis dans le cas (obstruction triviale, voir figure \ref{horocyclenondense}) 
de vecteurs $v$ 
dont l'un des demi-horocycles $(h^sv)_{s\ge 0}$ et $(h^sv)_{s\le 0}$ est 
dense dans $\mathcal{E}$, et l'autre sort de tout compact. 

Dans cette note, j'obtiens un panorama presque complet de la situation, 
sur des surfaces hyperboliques orient\'ees quelconques. 

Introduisons quelques d\'efinitions et notations. 
Soit $S$ une surface hyperbolique orient\'ee. Elle peut s'\'ecrire
$S=\Gamma \setminus \D$, o\`u $\Gamma$ est un groupe discret d'isom\'etries du disque 
hyperbolique $\D$, et $T^1S=\Gamma\setminus \D$. 
Si $u\in T^1 S$, et $\tilde{u}$ est un relev\'e de $u$ 
au fibr\'e unitaire tangent  $T^1\D$ 
du disque hyperbolique, nous notons  $u^-\in S^1$ (resp. $u^+$) l'extr\'emit\'e
n\'egative  (resp. positive) de la g\'eod\'esique $(g^t\tilde{u})_{t\in\R}$ 
d\'efinie par $\tilde{u}$ dans  
le bord \`a l'infini $S^1=\partial \D$ du disque hyperbolique.
L'ensemble limite $\Lambda_\Gamma\subset S^1$ est l'ensemble 
$\overline{\Gamma.o}\setminus\Gamma.o$ des points d'accumulation 
d'une orbite dans le bord. 
Dans toute cette note, $\Gamma$ est suppos\'e non-\'el\'ementaire, 
i.e. non virtuellement ab\'elien,
ce qui revient \`a dire
que $\#\Lambda_\Gamma=+\infty$. 
Nous aurons besoin d'orienter le bord $S^1$ dans le sens trigonom\'etrique. 

Un horocycle est un cercle (euclidien) de $\D$ tangent \`a $\partial \D$. 
Un horocycle instable est l'ensemble des vecteurs unitaires orthogonaux \`a un horocycle 
donn\'e, pointant vers l'ext\'erieur. 
Nous \'etudions ici le {\em flot horocyclique instable}. 
Ce flot, agissant sur $T^1\D$, 
a pour  orbites 
les horocycles instables, et il tourne les vecteurs 
d'une distance $|s|$ (pour la distance
riemannienne induite) sur l'horocycle instable, 
dans le sens des aiguilles d'une montre. 
Ce flot passe au quotient en le flot horocyclique instable de $T^1S$. 

Un vecteur $v\in T^1S$ est dit {\em radial} si sa g\'eod\'esique 
$(g^{-t}v)_{t\ge 0}$ revient 
infiniment souvent dans un compact (ici, $(g^t)_{t\in\R}$
 d\'esigne le flot g\'eod\'esique, qui
d\'eplace les vecteurs d'une distance $t$ le long de la
 g\'eod\'esique qu'ils d\'efinissent). 
Il est bien connu que si $v$ est radial, son horocycle 
$(h^sv)_{s\in\R}$ est dense dans l'ensemble non errant
 $\mathcal{E}$ du flot horocyclique, i.e.
l'ensemble $\mathcal{E}$ des vecteurs de $T^1S$ dont tout
 voisinage $V$ v\'erifie $h^{s_n}V\cap V\neq \emptyset $ pour une suite $s_n\to +\infty$. 
(C'est l'ensemble sur lequel se produit la dynamique int\'eressante.)

Le r\'esultat d'Hedlund s'\'etend alors: 

\begin{theo}\label{radial}Soit $S$ une surface
 hyperbolique non \'el\'ementaire orient\'ee. 
Soit $v\in T^1S$ un vecteur radial. 
Alors son demi-horocycle positif $(h^sv)_{s\ge 0}$ 
est dense dans l'ensemble non errant
$\mathcal{E}$ du flot horocyclique
si et seulement si $v^-$ n'est pas le premier point
 (dans le sens trigonom\'etrique) d'un
intervalle de $S^1\setminus \Lambda_\Gamma$. 
\end{theo}

Le r\'esultat ci-dessus est \'enonc\'e et d\'emontr\'e 
car sa preuve est simple et assez courte, 
et reprend
en les simplifiant des id\'ees d'Hedlund. 
Mais je d\'emontre ensuite un r\'esultat bien plus g\'en\'eral.

\begin{theo}\label{general}Soit $S$ une surface 
hyperbolique non \'el\'ementaire orient\'ee. 
Soit $v$ un vecteur  dont l'horocycle complet 
$(h^sv)_{s\in\R}$ est dense dans $\mathcal{E}$, et tel qu'il 
existe deux constantes $\Lambda>0$ et 
$0<\alpha_0\le \pi/2$, telles que le rayon g\'eod\'esique 
$(\pi(g^{-t}v))$, $t\ge 0$, croise une infinit\'e de g\'eod\'esiques 
ferm\'ees de longueur au plus $\Lambda$, avec un angle d'intersection 
au moins $\alpha_0$. 
Alors 
 les deux demi-orbites $(h^sv)_{s\ge 0}$ et 
$(h^sv)_{s\le 0}$ sont denses dans $\mathcal{E}$. 
\end{theo}

Un vecteur radial satisfait cette hypoth\`ese, 
car les g\'eod\'esiques ferm\'ees sont denses 
dans $\mathcal{E}$, et si $v$ est radial, 
$(g^{-t}v)$ revient infiniment souvent dans une
boule ferm\'ee de $\mathcal{E}$.  
On peut se convaincre que l'hypoth\`ese du th\'eor\`eme \ref{general} 
est alors v\'erifi\'ee, par exemple \`a l'aide d'une 
d\'ecomposition en pantalons de la surface. 

Il est \`a noter qu'une hypoth\`ese tr\`es proche 
est utilis\'ee dans un travail r\'ecent d'Omri
 Sarig \cite{Sarig} sur le flot horocyclique. 

On peut se demander si ce r\'esultat est  optimal.
Nous construisons un contre-exemple \`a un r\'esultat compl\`etement g\'en\'eral. 

\begin{theo}\label{contre-exemple} Il existe des surfaces hyperboliques non \'el\'ementaires 
orient\'ees contenant des
 vecteurs $v$ tels que $(h^sv)_{s\ge 0}$ est dense dans $\mathcal{E}$, 
$(h^sv)_{s\le 0}$ ne l'est pas, $v^-$ n'est pas une extr\'emit\'e d'un 
intervalle de $S^1\setminus\Lambda_\Gamma$, 
et $(g^{-t}v)_{t\ge 0}$ intersecte une infinit\'e de g\'eod\'esiques 
p\'eriodiques de longueur tendant vers l'infini. 

Suivant les exemples, l'angle d'intersection de $(g^{-t}v)_{t\ge 0}$
 avec ces g\'eod\'esiques
ferm\'ees peut \^etre uniform\'ement minor\'e ou tendre vers $0$. 
\end{theo}

Le paragraphe \ref{Preliminaires} est consacr\'e aux pr\'eliminaires, 
et les trois paragraphes 
suivants aux d\'emonstrations des trois r\'esultats ci-dessus. \\

Je remercie Yves Coud\`ene et Antonin Guilloux pour des discussions
 utiles li\'ees \`a ce travail.


\section{Pr\'eliminaires}\label{Preliminaires}


\subsection*{G\'eom\'etrie hyperbolique}
Le disque  hyperbolique $\D=D(0,1)$ est muni de la m\'etrique hyperbolique
$\frac{4dx^2}{(1-|x|^2)^2}$. Soit $o$ l'origine du disque, et $\pi:T^1\D\to \D$
la projection canonique. Le bord \`a l'infini de $\D$ est  
$S^1=\partial\D$. On note indiff\'eremment $d$ la distance 
riemannienne sur $\D$ et $T^1\D$. 

L'identification classique de $\D$ avec $\H=\R\times\R_+^*$ via 
l'homographie $z\mapsto i\frac{1+z}{1-z}$ permet d'identifier le groupe des 
isom\'etries pr\'eservant l'orientation de $\D$ avec 
 $PSL(2,\R)$ agissant par homographies sur $\H$. 
Cette action s'\'etend \`a $T^1\D$ (ou $T^1\H$), 
et devient simplement transitive; on identifie donc $T^1\D$ avec $PSL(2,\R)$.


Si $\Gamma\subset PSL(2,\R)$ est un groupe discret, son {\em ensemble limite} 
$\Lambda_\Gamma$ est d\'efini comme
$\Lambda_{\Gamma}=\overline{\Gamma. o}\setminus \Gamma. o \subset S^1$. 
C'est \'egalement le plus petit ferm\'e $\Gamma$-invariant non vide de $S^1$. 
Nous utiliserons souvent le fait que l'action de $\Gamma$ sur $\Lambda_\Gamma$ est 
minimale: pour tout $\xi\in\Lambda_\Gamma$, $\Gamma.\xi$ est dense dans $\Lambda_\Gamma$. 

Un point $\xi\in\Lambda_{\Gamma}$ est dit   {\em radial} 
s'il est limite d'une suite 
$(\gamma_n .o)$ de points de $\Gamma . o$ qui restent \`a distance hyperbolique born\'ee
du rayon g\'eod\'esique $[o\xi)$ reliant $o$ \`a $\xi$. 
Nous noterons 
 $\Lambda_{\rm rad}$ l'{\em ensemble limite radial}. 
L'ensemble des points de $\Lambda_\Gamma$ fix\'es par une isom\'etrie 
hyperbolique de $\Gamma$ est
inclus dans $\Lambda_{\rm rad}$.

Un {\em horocycle} de $\D$ est un cercle euclidien tangent \`a $S^1$; 
c'est \'egalement une ligne de niveau d'une fonction de Busemann.
Une {\em horoboule} est un disque (euclidien) bord\'e par un horocycle. 
Un point $\xi\in \Lambda_\Gamma$ est {\em horosph\'erique} si toute horoboule
centr\'ee en $\xi$ contient une infinit\'e de points de $\Gamma.o$. 
En particulier, $\Lambda_{\rm rad}$ est inclus dans l'ensemble 
des points limites horosph\'eriques, 
not\'e $\Lambda_{\rm hor}$. 

Une isom\'etrie de $PSL(2,\R)$ est dite {\em hyperbolique} si elle fixe exactement
deux points de $S^1$, {\em parabolique} si elle fixe exactement un point de $S^1$, 
et {\em elliptique} dans les autres cas. 
On note  $\Lambda_{\rm p}\subset \Lambda_{\Gamma}$ l'ensemble des points limites
{\em paraboliques}, i.e. les points fixes d'une isom\'etrie parabolique de $\Gamma$. 

Toute surface hyperbolique orient\'ee est le quotient 
  $S=\Gamma\backslash \D$ de $\D$ par un sous-groupe discret $\Gamma$ de $PSL(2,\R)$ 
sans \'el\'ement elliptique, et son fibr\'e unitaire tangent $T^1S=\Gamma\backslash T^1\D$
 s'identifie \`a 
$\Gamma\backslash PSL(2,\R)$.

Dans cette note, nous supposerons toujours $\Gamma$ {\em non \'el\'ementaire}, 
i.e.  $\#\Lambda_\Gamma=+\infty$.

Quand $S$ est compacte, alors $\Lambda_\Gamma=\Lambda_{\rm rad}=S^1$. 
La surface est dite {\em convexe-cocompacte} quand $\Gamma$ est finiment engendr\'e et
ne contient que des isom\'etries hyperboliques. 
Dans ce cas,  $\Lambda_{\Gamma}=\Lambda_{\rm rad}$ est strictement inclus dans $S^1$,
et  $\Gamma$ agit de mani\`ere cocompacte sur l'ensemble
$(\Lambda_{\Gamma}\times \Lambda_\Gamma)\setminus \mbox{Diagonale}\times\R\subset T^1\D$. 
Quand $S$  est de volume fini, ses bouts sont uniquement des pointes
 isom\'etriques \`a $\{z\in\mathbb{H}, |z|>1\}/<z\mapsto z+1>$, et 
$\Lambda_{\Gamma}=\Lambda_{\rm rad}\sqcup\Lambda_{\rm p}=S^1$. 

\subsection*{Flots g\'eod\'esique et horocyclique}

Une g\'eod\'esique hyperbolique de $\D$ est un diam\`etre
 ou un demi-cercle orthogonal \`a $S^1$. 
Un vecteur $v\in T^1\D$ est tangent \`a une unique g\'eod\'esique, et orthogonal \`a 
exactement deux horocycles passant par son point base, 
tangents \`a $S^1$ respectivement en $v^+$ et 
$v^-$. 
L'ensemble des vecteurs $w\in T^1\D$ tq $w^-=v^-$ et dont le point base appartient  
\`a ce dernier horocycle est {\em l'horocycle fortement instable}, ou 
vari\'et\'e fortement instable, de $v$.
On le note donc $W^{su}(v)=\{h^sv,s\in\R\}$.  
L'{\em horocycle fortement stable} est d\'efini de mani\`ere analogue. 

Le {\em flot g\'eod\'esique}   $(g^t)_{t\in\R}$  agit sur  $T^1\D$ 
en d\'epla\c cant un vecteur $v$ d'une distance $t$ 
le long de la g\'eod\'esique qu'il d\'efinit.
Dans l'identification de 
 $T^1\D$ avec $PSL(2,\R)$, ce flot correspond \`a l'action 
par multiplication \`a droite du sous-groupe
 \`a un param\`etre 
 $$
 \left\{a_t:=\left(\begin{array}{cc}
e^{t/2} & 0 \\
0 & e^{-t/2} \\
\end{array} \right),\,t\in\R
\right\}.
$$

Le {\em flot horocyclique instable}  $(h^s)_{s\in\R}$ agit sur $T^1\D$ 
en d\'epla\c cant un vecteur $v$ d'une distance $|s|$ le long de son horocycle 
fortement instable. Il y a deux orientations possibles pour un tel flot, 
et nous choisissons celle qui correspond \`a l'action 
\`a droite du groupe \`a un param\`etre 
 $$ 
\left\{n_s:=\left(\begin{array}{cc}
 1& 0 \\
s & 1 \\
\end{array} \right),\,s\in\R
\right\}
$$
sur $PSl(2,\R)$. Ce flot fait tourner les vecteurs 
le long de leur horocycle fortement instable, 
de sorte que $\{h^s v, \,s\in\R\}$ d\'ecrit tout l'horocycle fortement instable. 

De plus, pour tout  $s\in\R$ et $t\in\R$, ces flots 
g\'eod\'esique et horocyclique
satisfont la relation fondamentale suivante\,:
\begin{equation}\label{relationfondamentale}
 g^t\circ h^s=h^{se^t}\circ g^t\,.
 \end{equation}

\begin{rema}\label{orientation}\rm 
Avec notre choix d'orientation de $S^1$, quand $s\to +\infty$, si
 $u\in T^1\D$ et $u_s^+\in S^1$ est l'extr\'emit\'e positive de 
la g\'eod\'esique d\'etermin\'ee par  $h^su$, alors $u_s^+$ converge vers $u^-$, avec 
$u_s^+\ge u^-$ dans l'orientation trigonom\'etrique de $S^1$. 
\end{rema}

Ces deux actions \`a droite commutent avec l'action \`a gauche de $PSL(2,\R)$ sur 
lui-m\^eme par multiplication, de sorte qu'elles
 sont bien d\'efinies au quotient sur
$T^1 S\simeq \Gamma\backslash PSL(2,\R)$. 

\begin{defi}\label{nw} Soit $(\phi^t)_{t\in\R}$ un flot
agissant par hom\'eomorphismes sur un espace topologique $X$. 
L'{\em ensemble non errant} de $\phi$ est l'ensemble des 
 $x\in X$ tq pour tout voisinage $V$ de $x$, 
il existe une suite $t_n\to +\infty$ tq
 $\phi^{t_n} V\cap V\neq \emptyset$. 
\end{defi}


Nous renvoyons \`a \cite{Schapira} pour le lemme suivant. 

\begin{lemm}\label{nonwanderingsets}
L'ensemble non-errant du flot g\'eod\'esique agissant sur $T^1S$ est 
$$
\Omega:=\Gamma\backslash\left((\Lambda_\Gamma\times\Lambda_\Gamma)
\setminus\mbox{\rm Diagonal}\times\R\right).
$$
L'ensemble non errant du flot horocyclique agissant sur $T^1S$ est 
$$
\mathcal{E}:=\Gamma\backslash\left((\Lambda_\Gamma\times S^1)
\setminus\mbox{\rm Diagonal}\times\R\right).
$$ 
On a de plus
$$
\mathcal{E}=\cup_{s\in\R}h^s\Omega \quad \mbox{et } \quad
\mathcal{E}=\overline{\cup_{s\ge 0}h^s\Omega} 
$$ 
\end{lemm}




Rappelons que  $W^{su}(v)=\{h^s v,\,s\in\R\}$ est compact ssi $v^-\in \Lambda_{\rm p}$,
et dense in $\mathcal{E}$ ssi $v^-\in \Lambda_{\rm hor}$. 
Notons  $W^{su}_+(v)=\{h^s v,\,s\ge 0\}$ le demi horocycle positif de $v$.

Nous supposerons dans tout ce qui suit que $S^1$
 est orient\'e dans le sens trigonom\'etrique.

\subsection*{Trompettes}


Rappelons (voir par exemple \cite{Schapira}) les faits standards suivants. 
\begin{rema}\label{funnel}\rm Si la  surface $S$ vue comme $\Gamma\backslash\H$
 a une {\em trompette} isom\'etrique 
\`a $\{z\in\mathbb{H}, \,\mbox{Re}(z)\ge 0\}/\{z\mapsto az\}$, $a>1$,
la g\'eod\'esique  $Re(z)=0$ passe au quotient en une
 g\'eod\'esique p\'eriodique bordant la
trompette. Une g\'eod\'esique croisant cette g\'eod\'esique
 p\'eriodique et entrant dans la 
trompette n'en resortira jamais. 
En particulier, l'ensemble limite $\Lambda_\Gamma$ vu dans $\partial \H=\R\cup\{\infty\}$
n'intersecte pas $\R_+^*$.

Un  horocycle centr\'e dans $\R^*_+$, projet\'e sur $S$,
 restera dans la trompette sauf au plus pendant un intervalle de temps fini. 
Un horocycle centr\'e en $0$, vu dans $S$, aura une moiti\'e qui 
n'entrera pas dans la trompette, et l'autre c\^ot\'e qui restera dans
 la trompette sans jamais ressortir. 
\end{rema}

\begin{fait} \rm Dans le cas d'une surface hyperbolique 
g\'eom\'etriquement finie, 
les extr\'emit\'es d'un intervalle de $S^1\setminus \Lambda_\Gamma$ 
sont hyperboliques; 
ce sont les extr\'emit\'es de l'axe d'un relev\'e de 
la g\'eod\'esique bordant la trompette. 
Ce n'est pas n\'ecessairement le cas sur une surface quelconque 
(\footnote{On peut consid\'erer (exemple donn\'e par M. Peign\'e) 
pour s'en convaincre le groupe $\Gamma=<\alpha^n h \alpha^{-n},n\in\Z>$, 
o\`u $h$ et $\alpha$ sont deux isom\'etries hyperboliques en position Schottky, 
et $\alpha\notin \Gamma$, de sorte que ses points fixes $\alpha^-$ et $\alpha^+$
 sont les extr\'emit\'es d'un intervalle de $S^1\setminus\Lambda_\Gamma$, sans 
que cela corresponde \`a une trompette de la surface quotient.}).
\end{fait}

\begin{fait}\rm 
Si  $v^- \in \Lambda_{\Gamma}$ est la premi\`ere extr\'emit\'e d'un 
intervalle de 
  $S^1\setminus \Lambda_{\Gamma}$, alors $W^{su}_+(v)=\{h^sv, s\ge 0\}$ 
n'est pas dense dans $\mathcal{E}$. 
\end{fait}

\begin{figure}[ht!]
\begin{center}
\input{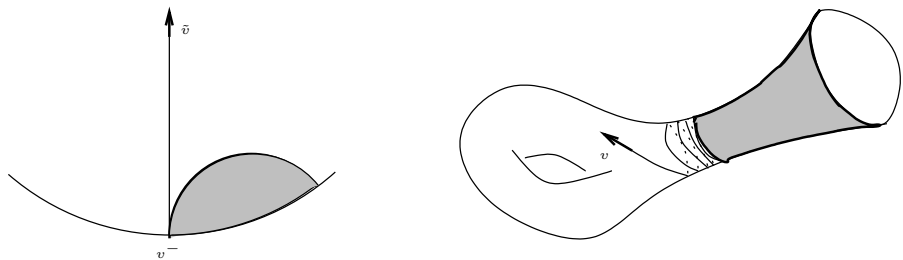}
\caption{Un vecteur $v$ dont l'horocycle positif $(h^sv)_{s\ge 0}$ 
n'est pas dense dans $\mathcal{E}$}
\label{horocyclenondense}
\end{center}
\end{figure}


\section{D\'emonstration du th\'eor\`eme \ref{radial}}

\subsection*{Points et vecteurs horocycliques \`a droite}

Si $v\in T^1\D$, soient $v^\pm$ ses extr\'emit\'es dans $\partial \D$, 
$Hor(v)\subset\D$ l'horoboule centr\'ee en $v^-$ et 
passant par le point base
$\pi(v)$ de $v$, et $Hor^+(v)\subset Hor(v)$ 
l'ensemble des points base des vecteurs
de $\cup_{t\ge 0}\cup_{s\ge 0}h^sg^{-t}v=\cup_{t\ge 0}\cup_{s\ge 0}g^{-t}h^sv$ 
(d'apr\`es la relation (\ref{relationfondamentale})).

\begin{defi} Si $v\in T^1\D$, et $\alpha>0$, 
le {\em c\^one} de largeur $\alpha$ autour de $v$
est l'ensemble 
$\mathcal{C}(v,\alpha)$ des points $x\in Hor(v)$ \`a distance (hyperbolique) 
au plus $\alpha$ du rayon g\'eod\'esique $(g^{-t}v)_{t\ge 0}$. Il s'agit de 
l'intersection de $Hor(v)$ avec un c\^one euclidien. 
\end{defi}

\begin{defi} Un vecteur $v\in T^1 S$ est {\em horocyclique \`a droite} 
s'il admet un relev\'e
 $\tilde{v}\in T^1 \D$, tel que pour tous  $\alpha>0$ et $D>0$, 
l'orbite $\Gamma.o$ intersecte la partie droite de l'horoboule 
 $Hor^+(g^{-D}\tilde{v})$ moins le c\^one 
$\mathcal{C}(g^{-D}\tilde{v}, \alpha)$. 

Un point $\xi\in \Lambda_\Gamma$ est {\em horocyclique \`a droite}
s'il existe $v\in T^1S$ horocyclique \`a droite tel que $\xi=v^-$. 
\end{defi} 

Pour m\'emoire rappelons qu'un vecteur $v$ est horosph\'erique 
si toutes les horoboules $Hor+(g^{-D}v)$ contiennent une infinit\'e de points
de $\Gamma.o$.

\begin{figure}[ht!]\label{Horocyclenondense}
\begin{center}\label{horocyclenondenseright-horocyclic}
\input{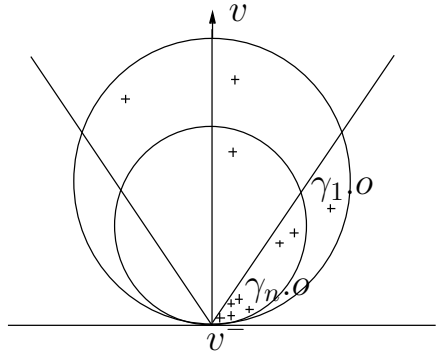}
\caption{Relev\'e d'un vecteur horocyclique \`a droite }
\end{center}
\end{figure}

Rappelons la proposition

\begin{prop}[\cite{Schapira}]\label{right_iff_dense}  
Soit $S$ une surface hyperbolique. 
Un vecteur $v\in T^1 S$ est horocyclique \`a droite si et seulement si 
   $(h^sv)_{s\ge 0}$ est dense dans  $\mathcal{E}$.
\end{prop}


Rappelons \'egalement la proposition suivante. 

\begin{prop}[\cite{Schapira}]\label{periodic} Soit $S$ une surface hyperbolique 
orient\'ee. Si $p\in\Omega$ est un vecteur p\'eriodique pour le flot g\'eod\'esique, 
son demi-horocycle positif
 $(h^s(p))_{s\ge 0}$ est
 dense dans l'ensemble non-errant  $\mathcal{E}$ du flot 
horocyclique ssi 
  $p^-$ n'est pas la premi\`ere extr\'emit\'e 
d'un intervalle de  $S^1\setminus\Lambda_{\Gamma}$. 
\end{prop}

\begin{theo}[Hedlund, \cite{Hedlund}, thm 4.2 ] 
Soit $S=\Gamma\setminus\D$ une surface hyperbolique orient\'ee de premi\`ere 
esp\`ece, i.e. tq $\Lambda_\Gamma=S^1$. Soit
 $v\in T^1 S$ tq  $(g^{-t}v)_{t\ge 0}$ revient 
infiniment souvent dans un compact, i.e. tq $v^-\in\Lambda_{\rm rad}$. Alors 
son demi-horocycle positif  $(h^s v)_{s\ge 0}$ est dense dans $T^1 S$.
\end{theo}

\subsection{D\'emonstration du th\'eor\`eme \ref{radial}}

Nous allons d\'emontrer qu'un vecteur  radial $v\in T^1S$ 
tel que $v^-$ n'est pas la premi\`ere extr\'emit\'e d'un 
intervalle de $S^1\setminus \Lambda_\Gamma$ est horosph\'erique \`a droite. 
Pour ne pas alourdir les notations, en l'absence d'ambiguit\'e, nous noterons 
encore $v$ un relev\'e de $v$ \`a $T^1\D$. 
Soient donc $R>0$ et $D>0$ deux constantes assez grandes. Montrons que 
$\#\Gamma.o\cap Hor^+(g^{-D}v)\setminus \mathcal{C}(v,R)=+\infty$. 

$\bullet$ Comme $v$ est radial, il existe $R_0$, 
et une suite $\gamma_n.o\to v^-$, telle que 
$d(\gamma_n.o,(v^-,v])\le R_0$. Soit $g^{-t_n}v$, $t_n\to +\infty$, 
une suite de vecteurs de la g\'eod\'esique 
\`a distance au plus $R_0$ de $\gamma_n.o$. 
En faisant agir $\gamma_n^{-1}$, on obtient 
une suite de vecteurs $\gamma_n^{-1}g^{-t_n}v$,
 tous dans le fibr\'e unitaire tangent $T^1B(o,R_0)$ de la boule $B(o,R_0)$. 
Quitte \`a extraire,  cette suite converge 
 vers un vecteur $w_\infty\in T^1B(o,R_0)$, d'extr\'emit\'es $w_\infty^\pm\in \partial\D$. 

Remarquons au passage (voir figure \ref{dessin-theoreme-radial}) 
que $\gamma_n^{-1}.v^-\to w_\infty^-$, et 
que le demi horocycle de $\D$ 
$\pi(\gamma_n^{-1}.\cup_{s\ge 0}h^sv)$ converge (au sens de la convergence de Hausdorff 
des ferm\'es de $\overline{\D}$)
 vers le demi-cercle $[w_\infty^-,w_\infty^+]$ du 
bord (orient\'e dans le sens trigonom\'etrique). \\

$\bullet$ Comme $v^-$ n'est pas la premi\`ere
 extr\'emit\'e d'un intervalle 
de $S^1\setminus \Lambda_\Gamma$, et que les $\gamma_n$ pr\'eservent l'orientation, 
on v\'erifie ais\'ement que $w_\infty^-$ n'est pas non plus la premi\`ere extr\'emit\'e 
d'un intervalle de $S^1\setminus\Lambda_\Gamma$. \\

$\bullet$ Choisissons maintenant une isom\'etrie
 hyperbolique $h\in \Gamma$, 
de points fixes $h^\pm\in \Lambda_\Gamma$ tous 
deux situ\'es dans l'intervalle 
$[w_\infty^-,w_\infty^+]$, \`a droite donc de $w_\infty^-$, proches de $w_\infty^-$. 
L'axe de $h$ se projette sur $S$ en une g\'eod\'esique
 p\'eriodique. 
Soit $D_0>0$ la distance de cette g\'eod\'esique au
 projet\'e de l'origine $o$ du disque $\D$. 

Quitte \`a conjuguer $h$ par une autre isom\'etrie hyperbolique ayant l'un de
 ses points fixes \`a droite de $w_\infty^-$, il est possible de supposer que 
l'axe $(h^-,h^+)$ de l'isom\'etrie est \`a distance sup\'erieure \`a $2R+R_0+D_0$ 
de la g\'eod\'esique $(w_\infty^-,w_\infty^+)$. 
\\

$\bullet$ 
Comme le demi-horocycle $\pi(\cup_{s\ge 0}h^s\gamma_n^{-1}v)$ converge vers 
$[w_\infty^-,w_\infty^+]$ (au sens de la convergence
 de Hausdorff des ferm\'es de $\overline{\D}$),
 pour $n$ assez grand, ce demi-horocycle intersecte
 la g\'eod\'esique $(h^-h^+)$ 
en deux points $x_n^-,x_n^+$, avec 
$d(x_n^-,x_n^+)\to \infty$. 

On voit \'egalement que le point, disons $y_n$, 
du segment g\'eod\'esique 
$[x_n^-,x_n^+]$ le plus loin de l'horocycle 
$\pi(\cup_{s\ge 0}h^s\gamma_n^{-1}v)$ 
voit sa distance \`a cet horocycle tendre vers l'infini avec $n$. 
Or ce point est \`a distance au plus $D_0$ d'un
 point de l'orbite $\Gamma.o$, 
que nous notons $g_n.o$. 
Autrement dit, pour $n$ assez grand, $g_n.o$ 
est  dans l'horoboule $Hor^+(g^{-D}\gamma_n^{-1}v)$. 

Le point $y_n$ est sur $(h^-h^+)$, donc
\`a distance au moins $2R+R_0+D_0$ de la g\'eod\'esique
 $(w_\infty^-,w_\infty^+)$. 
On en d\'eduit que $g_n.o$ est \`a distance au moins 
$2R$ de cette m\^eme g\'eod\'esique. 
Pour $n$ assez grand, $g_n.o$ est donc \`a distance au moins $R$ 
de la g\'eod\'esique $(\gamma_n^{-1}.v^-,\gamma_n^{-1}.v^+)$.

\begin{figure}[ht!]
\begin{center}
\input{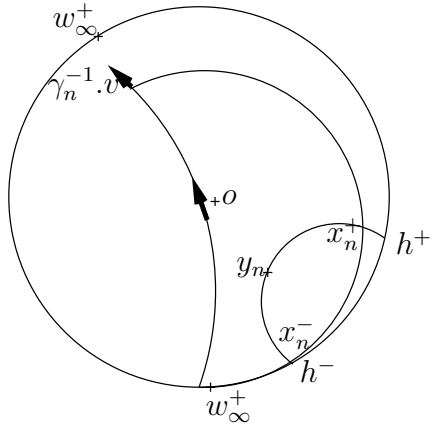}
\caption{D\'emonstration du th\'eor\`eme \ref{radial}}
\label{dessin-theoreme-radial}
\end{center}
\end{figure}

$\bullet$ Revenons par $\gamma_n$: notons $h_n^\pm=\gamma_n.h^\pm$, 
et $h_n=\gamma_n\circ h\circ \gamma_n^{-1}$ l'isom\'etrie hyperbolique associ\'ee. 
Les points $\gamma_n.g_n.o$, pour $n$ suffisamment grand, sont tous dans 
l'horoboule $Hor^+(g^{-D}v)$ sans \^etre dans le c\^one $\mathcal{C}(v,R)$.
 C'est bien le r\'esultat voulu.


\section{Et plus g\'en\'eralement}

Avant de commencer la preuve du th\'eor\`eme \ref{general}, 
rappelons quelques lemmes classiques de g\'eom\'etrie hyperbolique, 
dont on pourra trouver des \'enonc\'es dans \cite{GH} ou \cite{CDP} par exemple. 

\begin{lemm}\label{triangle} 
Soit $(a,b,c)$ un triangle hyperbolique (\'eventuellement infini). 
Si l'angle au sommet $a$ est minor\'e par $\alpha_0>0$, 
alors il existe une constante $C(\alpha_0)>0$, telle que
\begin{enumerate}
\item on a 
$d(a,b)+d(b,c)-C(\alpha_0)\le d(b,c)\le d(a,b)+d(a,c)$,
\item la distance de $a$ \`a $[b,c]$ est major\'ee par $C(\alpha_0)$.
\end{enumerate}
\end{lemm}

La r\'eciproque \`a ce lemme est \'egalement vraie: 

\begin{lemm}\label{triangle-reciproque}
Soit $k>0$. Il existe des constantes $\alpha(k)>0$, 
$d(k)>0$ et $C(k)=C(\alpha(k))$ telles
que si $(a,b,c)$ est un triangle hyperbolique 
(\'eventuellement infini) v\'erifiant 
$d(a,[bc])\le k$ et $d(b,c)\ge d(k)$, alors 
l'angle au sommet $a$ du triangle $(a,b,c)$
 est au moins \'egal \`a $\alpha(k)$, et on a 
$$
d(a,b)+d(b,c)-C(k)\le d(b,c)\le d(a,b)+d(a,c)\,.
$$

\end{lemm}

Enon\c cons un dernier lemme, \'egalement utile, sous la forme dont nous aurons besoin. 

\begin{lemm}\label{triangle-fin} Il existe une constante $\delta>0$, ne d\'ependant que de la 
g\'eom\'etrie de $\D$, t.q. pour tout 
$\xi\in \partial \D$ et $p,q\in\D$ t.q $\beta_{\xi}(p,q)=0$, on peut trouver un
 <<~triangle int\'erieur~>> $\alpha,\beta,\gamma$ v\'erifiant 
$\alpha\in (\xi,q]$, $\beta\in(\xi,p]$, $\gamma\in[p,q]$, $d(\alpha,\beta)\le \delta$, 
     $d(\alpha,\gamma)\le\delta$, $d(\beta,\gamma)\le \delta$, et de plus
$\beta_{\xi}(\alpha,\beta)=0$, $d(\alpha,q)=d(\gamma,q)$, $d(\beta,p)=d(\gamma,p)$. 

Dans cette situation, on a \'egalement 
$d(p,q)=2d(p,\gamma)=2d(p,\beta)=2d(q,\gamma)=2d(q,\alpha)$.  
\end{lemm}




D\'emontrons maintenant le th\'eor\`eme \ref{general}. 


\begin{proof} L'id\'ee est la suivante. 
Soit $\gamma$ une isom\'etrie hyperbolique correspondant
 \`a l'une des g\'eod\'esiques p\'eriodiques, de longueur au plus $\Lambda$, 
 crois\'ees par $(g^{-t}v)_{t\ge 0}$. 
Cette isom\'etrie d\'eplace les points de l'orbite $\Gamma.o$. 
Etant donn\'ees deux constantes $D>0$ et $R>0$, cette isom\'etrie it\'er\'ee 
convenablement devrait envoyer un point de 
$\Gamma.o\cap Hor^\pm(g^{-D}v)\setminus \mathcal{C}(v,R)$ 
dans $\Gamma.o\cap Hor^\mp(g^{-D+Cste}v)\setminus\mathcal{C}(v,R-Cste')$.

Introduisons quelques notations. 
On rel\`eve $v$ sur $T^1\D$, et on le note encore $v$. 
Soit $y_0\in \Gamma.o$ un point de 
$Hor^\pm(g^{-D}v)\setminus \mathcal{C}(v,R)$, et soit $x_0$
son projet\'e sur l'axe $(\gamma^-,\gamma^+)$. 
Notons $x_n=\gamma^n(x_0)$ et $y_n=\gamma^n(y_0)$. 
Supposons que $\gamma^-\ge v^-\ge \gamma^+$, 
sur le cercle orient\'e dans le sens trigonom\'etrique,
 de sorte que $\gamma$ va avoir tendance \`a 
envoyer des points de $Hor^+$ vers $Hor^-$, et non l'inverse. 
Supposons par cons\'equent que $v^-$ est horosph\'erique \`a droite,
 et montrons qu'il est horosph\'erique \`a gauche. \\


Fixons d'abord $D,R$, grands devant toutes les constantes de l'\'enonc\'e,  et des
lemmes de la preuve: la borne $\Lambda$ sur les longueurs des g\'eod\'esiques 
p\'eriodiques crois\'ees, l'angle $\alpha_0$ qui minore l'angle d'intersection entre
$(g^{-t}v)$ et ces g\'eod\'esiques p\'eriodiques, et les
 constantes $C(\alpha_0)$ du lemme \ref{triangle} et
 $C(k),d(k)$ du lemme \ref{triangle-reciproque} pour $k=C(\pi/2)$  ci-dessus 
et la constante $ C_2(\alpha_0)$ 
apparaissant ci-dessous. 

Fixons ensuite un point $y_0\in \Gamma.o$ 
de $Hor^+(g^{-D}v)\setminus \mathcal{C}(v,R)$, 
et choisissons pour finir une isom\'etrie $\gamma$ telle que 
l'axe $(\gamma^-\gamma^+)$ de $\gamma$ se projette en une 
g\'eod\'esique ferm\'ee intersectant $(g^{-t}v)_{t\ge 0}$, 
telle que l'angle sur $\D$ entre l'axe de $\gamma$ et
 $(g^{-t}v)$ est au moins $\alpha_0$,  
et telle que $y_0$ est dans la composante connexe born\'ee 
de $Hor(v)\setminus(\gamma^-\gamma^+)$. 
Comme il existe une infinit\'e de telles g\'eod\'esiques $(\gamma^-\gamma^+)$ 
de plus en plus loin de $y_0$, nous supposerons que la distance de $y_0$ \`a 
$(\gamma^-\gamma^+)$ est 
sup\'erieure \`a la constante $d(k)$ du lemme \ref{triangle-reciproque}, 
pour $k=C(\pi/2)$, constante donn\'ee par le lemme \ref{triangle}. 

Appelons $w=h^sv$ le vecteur de l'horocycle de $v$ tel que 
$(g^{-t}w)$ intersecte $(\gamma^-\gamma^+)$ orthogonalement. 
Notons respectivement $I_v$ et $I_w$ les 
points d'intersection de $(g^{-t}v)$ (resp. $(g^{-t}w)$)
avec $(\gamma^-\gamma^+)$. \\

\begin{figure}[ht!]
\begin{center}
\input{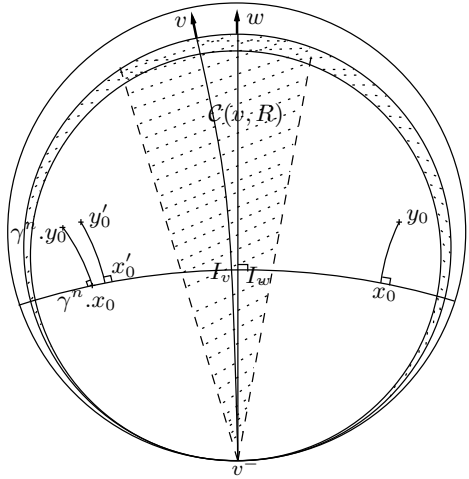}
\caption{D\'emonstration du th\'eor\`eme \ref{general}}
\end{center}
\end{figure}

\begin{lemm}\label{distance-majoree}
Si l'angle en $I_v$ entre $(\gamma^-\gamma^+)$ et $(g^{-t}v)_{t\ge 0}$ 
est minor\'e par $\alpha_0$, la distance entre $I_v$ et $I_w$
 est born\'ee par $C(\alpha_0)$. 
\end{lemm}

\begin{proof} En consid\'erant le triangle infini $v^-,I_v,I_w$, 
par le lemme \ref{triangle}, 
on obtient $d(I_v,(v^-I_w))\le C(\alpha_0)$. 
Comme l'angle en $I_w$ entre $(I_vI_w)$ et $(v^-I_w)$ vaut $\pi/2$, 
cette distance est \'egale \`a $d(I_v,I_W)$. Le lemme est d\'emontr\'e.
\end{proof}

\begin{lemm}\label{lemme-cle} Soit $x_0$ le projet\'e de $y_0$ sur $(\gamma^-\gamma^+)$, 
et $x_n=\gamma^nx_0$, $y_n=\gamma^ny_0$. 
Il existe des constantes $C_2(\alpha_0)$ et $R(\alpha_0)$, 
ne d\'ependant que de $\alpha_0$, telle que si $d(x_n,I_w)\ge R(\alpha_0)$, 
et $d(x_n,I_v)\ge R(\alpha_0/4)$, alors on a 
\begin{eqnarray*}
d(y_0,x_0)+d(x_n,I_w)-C_2(\alpha_0)& \le & d(y_n,(g^{-t}w)_{t\ge 0}) \le  d(y_0,x_0)+d(x_n,I_w)\\
 d(y_0,x_0)+d(x_n,I_v)-C_2(\alpha_0)& \le & d(y_n,(g^{-t}v)_{t\ge 0})  \le  d(y_0,x_0)+d(x_n,I_v)
\end{eqnarray*}
Par cons\'equent, on a 
$$
\left|d(y_n,(g^{-t}w)_{t\in\R})-d(y_n,(g^{-t}v)_{t\in\R})\right|\le 
d(I_v,I_w)+C_2(\alpha_0)\le C_1(\alpha_0)+C_2(\alpha_0)\,.
$$
\end{lemm}

\begin{proof} Le troisi\`eme encadrement d\'ecoule directement des deux premiers 
et du lemme \ref{distance-majoree}. Il suffit de d\'emontrer les deux premiers encadrements. 

Dans les deux cas, les majorations de droite d\'ecoulent
 imm\'ediatement de l'in\'egalit\'e triangulaire standard, et 
du fait que $d(x_n,y_n)=d(x_0,y_0)$.

Le triangle $y_n,x_n,I_w$ est rectangle en $x_n$,
 de sorte que par le lemme \ref{triangle}, 
$d(y_n,I_w)\ge d(y_n,x_n)+d(x_n,I_w)-C(\pi/2)=d(y_0,x_0)+d(x_n,I_w)-C(\pi/2)$. 
De la m\^eme mani\`ere, $d(y_n,I_v)\ge d(y_0,x_0)+d(x_n,I_v)-C(\alpha_0)$. 

Il reste \`a v\'erifier que la distance de $y_n$ \`a $(g^{-t}w)_{t\ge 0}$ 
(resp. $(g^{-t}v)_{t\ge 0}$), est, \`a des constantes uniformes pr\`es, 
r\'ealis\'ee par $d(y_n,I_w)$ (resp. $d(y_n,I_v)$).

Soit $p$ le projet\'e de $y_n$ sur $(g^{-t}w)_{t\in\R}$. 
Supposons que $d(x_n,I_w)\ge C(\alpha_0/4)$.  
Le lemme \ref{triangle} dans le triangle
$y_n, I_w, x_n$ implique alors que l'angle en $I_w$ entre $[I_w,x_n]$ et $[I_w,y_n]$ est 
inf\'erieur \`a $\alpha_0/4$. 

Alors, l'angle en $I_w$ entre $[I_wp]$ et $[I_w,y_n]$ est sup\'erieur \`a $\pi/2-\alpha_0/4>0$. 
Le lemme \ref{triangle} dans le triangle $(p,I_w,y_n)$ donne alors 
$d(y_n,p)\ge d(y_n, I_w)+d(I_w,p)-C(\pi/2-\alpha_0/4)$. 
Par cons\'equent, on a d\'emontr\'e que 
$d(y_n,(g^{-t}w)_{t\ge 0})=d(y_n,p)\ge d(y_n,I_w)-C(\pi/2-\alpha_0/4)\ge d(y_0,x_0)+d(x_n,I_w)-C(\pi/2)-C(\pi/2-\alpha_0)$.

Le m\^eme raisonnement en rempla\c cant $I_w$ par $I_v$ donne 
$d(y_n, (g^{-t}v)_{t\ge 0})\ge d(x_0,y_0)+d(x_n,I_v)-C(\alpha_0)-C(3\alpha_0/4)$. 
Ceci conclut la preuve du lemme. 
\end{proof}


Achevons maintenant la d\'emonstration du th\'eor\`eme. 
Rappelons que $y_0\in \Gamma.o$ est dans $Hor^+(g^{-D}v)\setminus \mathcal{C}(v,R)$, et
que nous souhaitons d\'emontrer que pour $n$ judicieusement choisi, $\gamma^n.y_0$ est
dans $Hor^-(g^{-D\pm cste}v)\setminus \mathcal{C}(v,R\pm Cste)$. 
Rappelons \'egalement que $w$ est le vecteur de $(h^sv)_{s\in\R}$ tq $(v^-w^+)$ est
orthogonale \`a $(\gamma^-\gamma^+)$. 

Soient $y'_0$ et $x'_0$ les sym\'etriques respectivement de 
$y_0$ et $x_0$ par rapport \`a la g\'eod\'esique $(v^-w^+)$. 
Comme les it\'er\'es $\gamma^n.x_0$ sont ses translat\'es d'une 
distance $l(\gamma)\le \Lambda$ sur l'axe $(\gamma^-,\gamma^+)$, 
il existe $n\ge 1$, tel que $d(\gamma^nx_0,x'_0)\le \Lambda$, et 
$d(\gamma^nx_0,x_0)\ge d(x'_0,x_0)$. 
Par sym\'etrie, on a $d(y'_0,\partial Hor(v))=d(y_0, \partial Hor(v))$ 
et $d(y'_0,(g^{-t}v)_{t\ge 0})=d(y_0,(g^{-t}v)_{t\in\R})$.

Notons $\mathcal{H}=Hor(v)$ et 
comparons d'abord $d(y_0,\partial \mathcal{H})\ge D$ et $d(y_n,\partial\mathcal{H})$. 
Par sym\'etrie, on a bien s\^ur $d(y'_0,\partial\mathcal{H})\ge D$. 
Comme $x_n=\gamma^n.x_0$ et $x'_0$ sont \`a distance au plus $\Lambda$,
on en d\'eduit que $|d(x_n,\partial\mathcal{H})-d(x'_0,\partial\mathcal{H})|\le \Lambda$.
Nous allons chercher \`a minorer 
$d(y_n,\partial\mathcal{H})-d(y'_0,\partial\mathcal{H})$ par
une constante uniforme, ce qui donnera $d(y_n,\partial\mathcal{H})\ge D-cste$ 
pour $D$ assez grand.


Notons respectivement $q_n$ et $q'$ les projet\'es de 
$y_n=\gamma^n.x_0$ et de $y'_0$ sur $\partial\mathcal{H}$. 
Si $d(y_n,\partial\mathcal{H})\ge d(y'_0,\partial\mathcal{H})\ge D$, c'est parfait. 
Supposons donc que $d(y_n,\partial\mathcal{H})\le  d(y'_0,\partial\mathcal{H})$, 
et montrons que 
cette distance ne peut pas \^etre trop petite devant $D$. 

\begin{figure}[ht!]
\begin{center}
\input{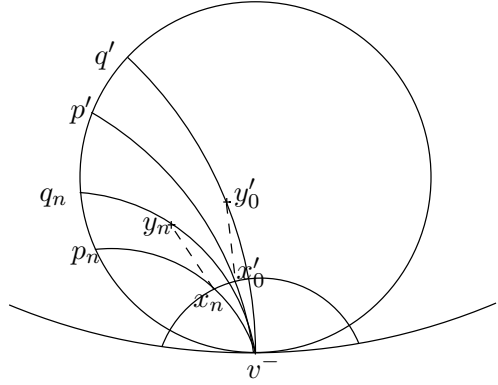}
\caption{D\'emonstration du th\'eor\`eme \ref{general}}
\end{center}
\end{figure}
Le triangle $(y_n,x_n,v^-)$ a un angle sup\'erieur \`a $\pi/2$ en $x_n$ (car $x_n$ est 
le projet\'e de $y_n$ sur $(\gamma^-,\gamma^+)$ qui intersecte $[y_n,v^-)$). 
D'apr\`es le lemme \ref{triangle}, on a $d(x_n,[y_n,v^-))\le C(\pi/2)$.
En consid\'erant le triangle $(q_n,x_n,v^-)$, 
on en d\'eduit en particulier que la distance de $y_n$ au c\^ot\'e $[q_n,x_n]$ 
v\'erifie $d(y_n,[q_n,x_n])\le C(\pi/2)$. 
On a suppos\'e $d(y_n,x_n)=d(y_n,(\gamma^-,\gamma^+))$ grande devant $k(C(\pi/2))$. 
Le lemme \ref{triangle-reciproque} donne alors 
$d(x_n,y_n)+d(y_n,q_n)-C(k(C(\pi/2)))\le d(x_n,q_n)\le d(x_n,y_n)+d(y_n,q_n)$.
Le m\^eme raisonnement s'applique bien s\^ur en
 rempla\c cant $x_n,y_n,q_n$ par $x'_0,y'_0,q'$. 
On en d\'eduit alors (en utilisant le fait que $d(x_n,y_n)=d(x'_0,y'_0)$) que 
$$
d(y_n,\partial\mathcal{H})-d(y'_0,\partial\mathcal{H})=
d(y_n,q_n)-d(y'_0,q')\ge d(x_n,q_n)-d(x'_0,q')-C(k(C(\pi/2)))\,.
$$

On sait par ailleurs  que 
$|d(x_n,\partial\mathcal{H})-d(x'_0,\partial\mathcal{H})|\le \Lambda$.
Il nous reste donc \`a comparer $d(x_n,q_n)$
 et $d(x_n,\partial\mathcal{H})$ d'une part, 
et $d(x'_0,q')$ \`a $ d(x'_0,\partial\mathcal{H})$ d'autre part. 
On a clairement $d(x_n,q_n)\ge d(x_n,\partial\mathcal{H})$. 
Montrons que $d(x'_0,q)$ est proche de $d(x'_0,\partial\mathcal{H})$. 

Consid\'erons le triangle $(p',q',v^-)$ et appliquons lui le lemme \ref{triangle-fin}. 
Avec les notations de ce lemme, le point $x'_0$ est sur $(v^-,p']$. Si 
$x'_0\in [\beta,v^-)$, alors il est ais\'e d'en d\'eduire que 
$|d(x'_0,p')-d(x'_0,q)|\le\delta$. 
Si $x'_0$  \'etait sur $]\beta,p']$,
 alors il existerait un point de $(\gamma^-\gamma^+)$ 
\`a l'int\'erieur du triangle $(\alpha,\beta,\gamma)$ plus proche de $y'_0$ que $x'_0$, ce 
qui contredirait la d\'efinition de $x'_0$. 
 
Tout ceci implique que 
$$
|d(x'_0,\partial\mathcal{H})-d(x'_0,q')|=|d(x'_0,p')-d(x'_0,q')|\le\delta\,.
$$
En rassemblant tous ces encadrements, on a d\'emontr\'e que 
\begin{eqnarray*}
d(y_n,\partial\mathcal{H})-d(y'_0,\partial\mathcal{H})
&\ge & d(x_n,\partial\mathcal{H})-d(x'_0,\partial\mathcal{H})-\delta-C(k(C(\pi/2)))\\
&\ge &-\Lambda-\delta-C(k(C(\pi/2)))\,.
\end{eqnarray*}
En particulier, on en d\'eduit le r\'esultat voulu, 
\`a savoir que pour tout $D>0$ assez grand, 
si $d(y_0,\partial\mathcal{H})\ge D$, alors 
$d(y_n,\partial\mathcal{H})\ge D-\Lambda-\delta-C(k(C(\pi/2)))$.


Reste \`a comprendre la distance de $y_n=\gamma^n.y_0$ 
\`a la g\'eod\'esique $(g^{-t}v)_{t\ge 0}$. 
On a suppos\'e $d(y_0,(g^{-t}v)_{t\ge 0})\ge R$. 
Si $y'_0$ est plus loin que $y_0$ de cette g\'eod\'esique, on a 
$d(y'_0,(g^{-t}v)_{t\in\R})\ge R$. 
Sinon, cela signifie que $y'_0$ est plus proche de $(g^{-t}v)_{t\ge 0}$ 
que de $(g^{-t}w)_{t\ge 0}$. 
Par sym\'etrie, et par le lemme \ref{lemme-cle},  on a 
\begin{eqnarray*}
d(y'_0,(g^{-t}w)_{t\ge 0})&=&d(y_0,(g^{-t}w)_{t\ge 0})\\
&\ge& d(y_0,(g^{-t}v)_{t\ge 0})-C(\alpha_0)-C_2(\alpha_0)\ge R-C(\alpha_0)-C_2(\alpha_0).
\end{eqnarray*}
D'apr\`es le lemme \ref{lemme-cle}, 
$d(y_n,(g^{-t}v)_{t\ge 0})\ge d(y_n,(g^{-t}w)_{t\ge 0})-C(\alpha_0)-C_2(\alpha_0)$. 

Par d\'efinition de $y'_0$ et $y_n$, $y_n$ est plus loin que $y'_0$ de $(g^{-t}w)_{t\in\R}$. 
On d\'eduit de tout cela que $d(y_n,(g^{-t}v)_{t\ge 0}\ge R-2C(\alpha_0)-2C_2(\alpha_0)$. \\

Autrement dit, si $y_0$ est dans $Hor^+(g^{-D}v)\setminus\mathcal{C}(v,R)$, alors le point
$\gamma^n y_0$ est dans 
$Hor^-(g^{-D+\Lambda+\delta+C(k(C(\pi/2)))}v)\setminus \mathcal{C}(v,R-2C(\alpha_0)-2C_2(\alpha_0))$.
En particulier, si $v$ est horosph\'erique \`a gauche, 
alors il est aussi horosph\'erique \`a droite. 
Compte tenu de la proposition \ref{right_iff_dense}, 
ceci d\'emontre que sous les hypoth\`eses du th\'eor\`eme \ref{general}, 
$(h^sv)_{s\ge 0}$ et $(h^sv)_{s\le 0}$ sont simultan\'ement
 denses (ou non) dans $\mathcal{E}$.

Ceci d\'emontre le th\'eor\`eme. 
\end{proof}


\section{Un contre-exemple pour finir}

On d\'emontre le r\'esultat suivant:\\

\noindent
{\bf Th\'eor\`eme \ref{contre-exemple}} {\em  
Il existe des surfaces hyperboliques contenant des
 vecteurs $v$ tels que $(h^sv)_{s\ge 0}$ est dense dans $\mathcal{E}$, 
$(h^sv)_{s\le 0}$ ne l'est pas, $v^-$ n'est pas une extr\'emit\'e d'un 
intervalle de $S^1\setminus\Lambda_\Gamma$, 
et $(g^{-t}v)_{t\ge 0}$ intersecte une infinit\'e de g\'eod\'esiques 
p\'eriodiques de longueur tendant vers l'infini, 
suivant un angle qui peut suivant les constructions tendre vers $0$, 
ou \^etre uniform\'ement minor\'e. \\}

L'id\'ee de la construction est la suivante. On prend $v^-=\infty$, $v^+=0$. 
Et on \'etudie l'orbite de $o=i$. 

On choisit sur $\R_+$ des demi-cercles $C^+_n$ tous tangents, 
de rayon euclidien born\'e, disons $1$, 
centr\'es en $2n+1$, $n\ge 0$,  allant jusqu'\`a l'infini. 
On choisit sur $\R_-$ des demi-cercles $C_n^-$ tangents deux \`a deux, 
centr\'es en $-x_n$ de rayon $R_n\to +\infty$. Par une r\'ecurrence imm\'ediate, 
on a $x_1=R_1$, et $x_{n}=\sum_{k=0}^{n-1}2R_k+R_n$.
On choisit des isom\'etries hyperboliques $\gamma_n$ 
de longueurs de plus en plus grandes, de points fixes $\gamma_n^-=2n+1$ et $\gamma_n^+=x_n$, 
et envoyant $C_n^+$ sur $C_n^-$. 

Un argument de ping pong classique donne le lemme suivant. 

\begin{lemm} Le groupe engendr\'e par 
les $\gamma_n$ est un groupe discret, groupe libre (de Schottky)
\`a une infinit\'e de g\'en\'erateurs. 
\end{lemm}

\begin{rema}\rm Il n'est pas clair si le groupe ainsi construit 
v\'erifie $\Lambda_\Gamma=S^1$. 
Si ce n'est pas le cas, il serait int\'eressant de construire un 
contre-exemple v\'erifiant $\Lambda_\Gamma=S^1$. 
\end{rema}

Maintenant,  il est assez clair que le vecteur 
$v$ bas\'e en $o=i$ avec $v^-=\infty$, $v^+=0$,
ne peut pas \^etre horosph\'erique \`a gauche 
(\footnote{la terminologie {\em horosph\'erique \`a droite ou \`a gauche} 
peut sembler absurde sur cet exemple, mais ... elle est adapt\'ee 
pour tous les points de l'axe r\'eel dans le bord \`a l'infini.}),
car l'orbite de $o=i$ ne rencontre pas du tout
$Hor^-(v)$. 

Il est \'egalement clair que si $\xi\in\Lambda_\Gamma$, $\xi\neq v^\pm$,
 $\gamma_n\xi$ tend vers 
$v^-$ par la droite si $n\to +\infty$, 
et $\gamma_n^{-1}\xi$ tend vers $v^-$ par la gauche. 
Autrement dit, $v^-$ n'est pas dans le bord de $S^1\setminus \Lambda_\Gamma$. 

\begin{figure}[ht!]
\begin{center}
\input{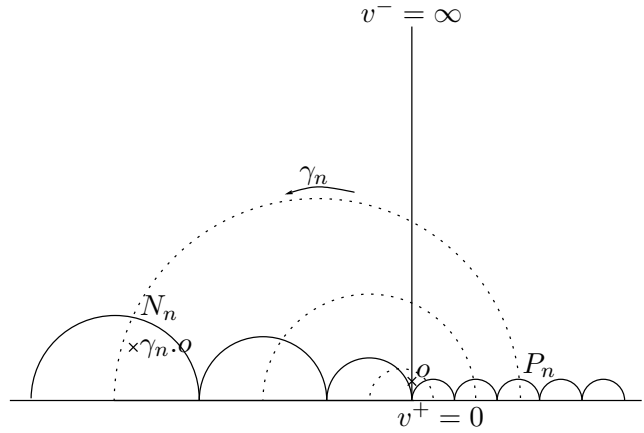}
\caption{D\'emonstration du th\'eor\`eme \ref{contre-exemple}}
\end{center}
\end{figure}

Il nous reste \`a v\'erifier que pour  un choix convenable de $x_n$ et de $r_n$, 
$v^-$ est horosph\'erique \`a droite. 
Pour cela, montrons que l'ordonn\'ee de $\gamma_n.o$ tend vers $+\infty$ 
(l'abscisse tend vers $-\infty$ par construction.)

Notons $z_n$ le point de coordonn\'ees $(2n+1,1)$, et 
$P_n$ le point d'intersection de l'axe de $\gamma_n$ et du cercle $C_n^+$, 
soit encore l'intersection des deux demi-cercles d'\'equations 
$\displaystyle (x-(2n+1))^2+y^2=1$ et 
$\left(x-\left(\frac{2n+1-x_n}{2}\right)\right)^2+y^2=\left(\frac{2n+1+x_n}{2}\right)^2$. 
Il a pour coordonn\'ees 
$\left(2n+1 -\frac{1}{x_n+2n+1}, \sqrt{1-\frac{1}{(x_n+2n+1)^2}}\right)$, 
de sorte que la distance hyperbolique de $P_n$ \`a $z_n$ v\'erifie
 $d(P_n,z_n)\to 0$ quand $n\to \infty$, et que pour 
$n$ assez grand (ind\'ependemment m\^eme du choix de $x_n$), $d(P_n,z_n)\le 1$. 
Par ailleurs, un calcul classique donne $d(o,z_n)\sim 2\ln n$ quand $n\to +\infty$. 
Par cons\'equent, $d(o,P_n)\le 3 \ln n$ quand $n\to +\infty$. 

L'image de $P_n$ par $\gamma_n$ est le point d'intersection $N_n=\gamma_n.P_n$ 
du demi-cercle $C_n^-$ et de l'axe de $\gamma_n$, d'\'equations respectives  $(x+x_n)^2+y^2=r_n^2$
et 
$(x-\frac{2n+1-x_n}{2})^2+y^2=\left(\frac{2n+1+x_n}{2}\right)^2$. 
Le point $N_n$ a donc pour coordonn\'ees 
$$
\left(-x_n+\frac{r_n^2}{2n+1+x_n}, r_n\sqrt{1-\frac{r_n^2}{(2n+1+x_n)^2}}\right)\,. 
$$

Maintenant, on sait que la distance de $\gamma_n.o$ \`a $N_n$ est
 au plus $3\ln n$ pour $n$ grand. 
On souhaite avoir des $\gamma_n.o$ de partie imaginaire aussi
 grande que possible. 
Or $\mbox{Im}(\gamma_n.o)\ge r_n\sqrt{1-\frac{r_n^2}{(2n+1+x_n)^2}}-3\ln n$. 

Observons que $x_n+r_n=2\sum_{k=1}^{n}r_k$. 
Si nous choisissons pour tout $k\ge 1$, $r_k=k$, 
nous obtenons $x_n+n=n(n+1)$, soit $x_n=n^2$, et 
$r_n\sqrt{1-\frac{r_n^2}{(2n+1+x_n)^2}}=n\sqrt{1-\frac{n^2}{(n+1)^4}}\sim n$ 
quand $n\to \infty$. 
Autrement dit, ce choix de $(r_n)_{n\ge 1}$ convient. 

Si nous choisissons pour tout $k\ge 1$, $r_k=\alpha^k$, $\alpha>1$,
 nous obtenons 
$x_n+\alpha^n=2\alpha(\alpha^n-1)/(\alpha-1)$, 
puis $x_n=\alpha^n\frac{\alpha+1}{\alpha-1}-\frac{2\alpha^{1-n}}{\alpha-1}$. 
Notons $y_n$ l'ordonn\'ee de $N_n$.
Une v\'erification imm\'ediate donne 
$y_n^2\sim \alpha^{2n}\frac{4\alpha}{(\alpha+1)^2}>>(3\ln n)^2$. 
Donc $\mbox{Im}(\gamma_n.o)\to +\infty$ quand $n\to \infty$, et ce
 choix de $(r_n)_{n\in\N}$ convient aussi.

\begin{rema}\rm Il est clair par construction qu'au quotient, 
 la g\'eod\'esique $(g^{-t}v)_{t\ge 0}$
 croise une infinit\'e de g\'eod\'esiques ferm\'ees, de 
longueur tendant vers $+\infty$, \`a savoir
les projections sur la surface quotient des axes des $\gamma_n$. 

Un calcul montre que l'angle $\theta_n$ entre la g\'eod\'esique $(v^-v^+)=i\R$, 
et l'axe de $\gamma_n$ v\'erifie 
$\cos\theta_n=\frac{x_n-(2n+1)}{x_n+2n+1}$. 
Dans les deux exemples ci-dessus,  on a $2n+1=o(x_n)$, 
d'o\`u $\cos \theta_n\to 1$ et $\theta_n\to 0$. 

Autrement dit, l'exemple ci-dessus ne v\'erifie 
aucune des deux hypoth\`eses du th\'eor\`eme 
\ref{general}.

On peut le modifier pour qu'il v\'erifie l'une des deux. 
Par exemple, la lectrice v\'erifiera ais\'ement que 
si les cercles $C_n^-$ sont inchang\'es, 
centr\'es en $-x_n$, et de rayon $R_n\to +\infty$, 
mais les cercles $C_n^+$ sont maintenant centr\'es en $+x_n$, 
et toujours de rayon $1$ (ils ne sont plus tangents deux \`a deux), 
alors la g\'eod\'esique $(v^-v^+)=i\R$ 
intersecte l'axe $(-x_n,x_n)$ de $\gamma_n$ orthogonalement. 
La distance de $o$ \`a $P_n$ est \'equivalente \`a $2\ln x_n$, 
donc $d(\gamma_n.o,N_n)\sim 2\ln x_n$, alors que l'ordonn\'ee 
du point $N_n$ vaut $y_n=r_n\sqrt{1-\frac{r_n^2}{4x_n^2}}$. 
Si $R_n=n$ et $x_n=n^2$, la partie imaginaire de $\gamma_n.o$ 
tend vers $+\infty$, de sorte que le point $+\infty$ est encore 
horosph\'erique \`a droite mais pas \`a gauche. 

Par ailleurs, on peut si on le souhaite rajouter des isom\'etries 
envoyant des cercles de hauteur born\'ee les uns sur les autres, 
pour <<~boucher les trous~>> entre les demi-cercles $C_n^+$. 
Il n'est pas s\^ur que cela donne un ensemble limite $\Lambda_\Gamma=S^1$, 
du fait du diam\`etre non born\'e des cercles $C_n^-$. 
\end{rema}


\begin{thebibliography}{99}












\bibitem[{\bf C-D-P}]{CDP} M. Coornaert, T. Delzant, A. Papadopoulos, {\em G\'eom\'etrie 
et th\'eorie des groupes}, in 'Les groupes hyperboliques de Gromov', Lecture Notes in Math. 1441, Springer Verlag, Berlin 1990. 


\bibitem[{\bf Da}]{Dalbo} F. Dal'bo {\em Topologie du feuilletage fortement
  stable}, Ann. Inst. Fourier (Grenoble) {\bf 50} (3) (2000), 981--993. 




\bibitem[{\bf E1}]{Eberlein1} P. Eberlein {\em Geodesic flows on negatively curved manifolds I},
Ann. of Maths (2) {\bf 95} (1972), 492-510. 

\bibitem[{\bf E2}]{Eberlein} P. Eberlein {\em Geodesic flows on negatively curved manifolds II},
 Trans. Amer. Math. Soc. {\bf 178} (1973), 57-82.


\bibitem[{\bf G-Ha}]{GH} \'E. Ghys, P. de la Harpe, {\em Sur les groupes hyperboliques d'apr\`es Mikhael Gromov} (Berne 1988), Progr. Math. vol. 83, 1-25, Birkha\"user Boston, Boston, MA, 1990. 

\bibitem[{\bf H}]{Hedlund} G.A. Hedlund {\em Fuchsian groups and transitive horocycles}, 
Duke Math. J. {\bf 2} (1936), 530-542.  










\bibitem[{\bf Sa}]{Sarig} O. Sarig 
{\em The horocycle flow and the Laplacian on hyperbolic surfaces of infinite genus}.
 Geom. Funct. Anal. 19  (2010),  no. 6, 1757-1812.

\bibitem[{\bf Sa-Scha}]{SScha} O. Sarig, B. Schapira 
{\em The generic points for 
the horocycle flow on a class of hyperbolic surfaces with infinite genus}, Int. Math. Res. Not. IMRN  (2008), Art. ID rnn 086, 37 pp.  


\bibitem[{\bf Scha}]{Schapira} B. Schapira, {\em Density and equidistribution of half-horocycles 
on a geometrically finite hyperbolic surface}, (2010) To appear in Journal of the LMS .






\end{thebibliography}
\end{document}